\documentclass[11pt]{article}
\usepackage{setspace}
\usepackage{amsthm,amsfonts,amsmath}
\usepackage{amssymb, amsthm, algorithm, algorithmic,thmtools,thm-restate,enumerate,verbatim,bm,color}
\definecolor{darkgreen}{rgb}{0.1,0.4,0.1}
\usepackage{graphicx}
\usepackage{enumerate}
\usepackage{verbatim} 
\usepackage{tikz}
\usetikzlibrary{arrows}
\usepackage[authoryear,round]{natbib}
\usepackage{algorithm}
\usepackage{algorithmic}

\allowdisplaybreaks
\usepackage{fullpage}
\newcommand{\bbr}{\mathbb{R}}  

\newcommand{\C}{\mathcal{C}}

\renewcommand{\labelenumi}{\theenumi}

\newcommand{\be}{\begin{equation}}
\newcommand{\ee}{\end{equation}}
\newcommand{\bew}{\begin{equation*}}
\newcommand{\eew}{\end{equation*}}

\newcommand{\E}{\mathbb{E}}
\newcommand{\EE}{\mathbb{E}}

\renewcommand{\P}{\mathbb{P}}

\newcommand{\df }{\stackrel{\Delta}{=}}

\declaretheorem[numberwithin=section]{theorem}
\declaretheorem[sibling=theorem]{lemma}
\declaretheorem[sibling=theorem]{proposition}

\declaretheorem[sibling=theorem]{assumption}
\declaretheorem[sibling=theorem]{remark}
\declaretheorem[sibling=theorem]{definition}
\declaretheorem[sibling=theorem]{example}
\makeatletter
\def\eqalign#1{%
 \null\,\vcenter{\openup\jot\m@th
  \ialign{\strut\hfil$\displaystyle{##}$&$\displaystyle{{}##}$\hfil
      &&\hfil$\displaystyle{##}$&$\displaystyle{{}##}$\hfil\crcr#1\crcr}}\,}

\makeatother

\newenvironment{proofof}[1]{{{\noindent \em {Proof of}} #1.}}{\hfill\rule{2mm}{2mm}}

\usepackage{algorithm}
\usepackage{algorithmic}
\newcommand{\INPUT}{\item[{\bf Input:}]}
\newcommand{\OUTPUT}{\item[{\bf Output:}]}
\newcommand{\ea}[1][]{e^{-#1 \alpha}}
\newcommand{\reward}{r(X_t)}
\newcommand{\enorm}{{\tilde \pi_N}}
\newcommand{\Proj}{\Pi}

\begin{document}
\author{Mohammad Mousavi\footnote{Corresponding author. Department of Management Science \& Engineering,
Stanford University, email: {\tt mousavi@stanford.edu},
web: {\tt www.stanford.edu/$\sim$mousavi}.},\; Peter W. Glynn\footnote{Department of Management Science \& Engineering,
Stanford University, email: {\tt glynn@stanford.edu},
web: {\tt www.stanford.edu/$\sim$glynn}.}} \normalsize

\date{\today} 
\thispagestyle{empty}
\title{\Large \textbf{Shape-constrained Estimation of Value Functions  \normalsize}} \normalsize
\maketitle

\setcounter{section}{0}

%
\begin{abstract}
We present a fully nonparametric method to estimate the value function, via simulation, in the context of expected infinite-horizon discounted rewards for Markov chains. Estimating such value functions plays an important role in approximate dynamic programming. We incorporate ``soft information'' into the estimation algorithm, such as knowledge of convexity, monotonicity, or Lipchitz constants. In the presence of such information, a nonparametric estimator for the value function can be computed that is provably consistent as the simulated time horizon tends to infinity. As an application, we implement our method on price tolling agreement contracts in energy markets. 
\end{abstract}
\newpage
\section{Introduction}
This paper is concerned with the estimation, via simulation, of value
functions in the context of expected infinite horizon discounted rewards
for Markov chains. Estimating such value functions plays an important role
in approximate dynamic programming 
and applied probability in general. In
many problems of practical interest, the state space is huge or even
continuous and the value function is computationally intractable. Therefore,
we need to approximate the value function. In this work, we develop a fully
non-parametric method to estimate the value function by incorporating shape
constraints, such as knowledge of convexity, monotonicity, or Lipschitz
constants.

The most common method employed to approximate the value function is
parametric approximate dynamic programming; see \citet{powell2007approximate} and \citet{bertsekas2007dynamic}. In
this method, the user specifies an ``approximation architecture'' (i.e. a
set of basis functions) and the algorithm then produces an approximation in
the span of this basis. Selecting the basis function is essential because
an inappropriate ``approximation architecture''
might cause unsatisfactory
results, and cannot be improved by additional sampling or computational
effort.

In contrast, we are proposing a fully ``non-parametric'' method to avoid the
difficulty of  choosing a correct approximation architecture. The general
idea is to take 
advantage of shape properties of the optimal value function in
estimating the function. A variety of control problems exist on continuous
state spaces for which convexity in the value function naturally arises.
For instance, in a linear transition system, if the reward function
in each stage is convex, then the value function is convex. Inventory
models represent a well-known example of this class of problems. Singular
stochastic control (\citet {sunil}) and partially observed Markov processes
(\citet{POSMP}) are two other subclasses of problems for which the value
function is convex.  As another example, \citet{American_option} show that
the American-style option is convex for a generalized 
Black--Scholes model.

Monotonicity properties have been studied in the literature for various
problems formulated as Markov decision processes. For instance, if the
reward is monotone and the chain is stochastically monotone, the value
function is 
monotone. In \citet{papadaki2007monotonicity}, the monotonicity
of the value function is studied in the case of the multi-product batch
dispatch problem. \citet[p. 267-268]{stokey1989recursive} presents general
conditions that guarantee the value function will be monotonic in the
underlying state variable. Discussions of monotonicity appear also in
\cite{serfozo1976monotone} and \cite{topkis1998supermodularity}.

In \citet{convex_prop_ms} and \citet{convex_prop_econ}, sufficient conditions
are provided on the transition probability of a stochastic dynamic programming
problem to ensure the shape properties of the 
value function.  The goal of our
work is to exploit this type of shape property to estimate the value function.

We suggest two methods for computing an approximation of the value function
for a fixed policy. In the first method, we estimate the value function
along a path by explicitly incorporating the shape constraint. For instance,
in the case that we know the value function is convex, we consider the set
of all convex functions which is a convex cone in the space of measurable
functions. Having a sample path of the underlining process, one can reach a
noisy observation of the value function. By projecting this noisy observation
to the cone of convex functions, we achieve an estimator for the value
function. Since this method requires only one sample path of the process,
it can be used in reinforcement learning applications.

The second method is based on estimating the value function by taking
advantage of the 
fixed point property of the value function in addition to 
the
shape constraint. The value function satisfies a specific linear system
of equations. Therefore, estimating the value function is possible by
approximating the fixed point of this system of equations over the cone of
convex functions. This fixed point can be obtained by iteratively projecting
onto the cone of convex functions. The simulation results show that the
second approach has reduced variance and provides more accurate 
estimators as compared
to the first approach.

The projection onto the cone of convex functions is possible by solving
a least square 
optimization problem. This optimization problem can be
interpreted as a multi-dimensional convex regression. Convex regression is
concerned with computing the best fit of a convex function to a dataset of
$n$ 
observations;
$$Y_i=f(X_i)+\epsilon_i$$
 for $i=1,\ldots,n$. Convex regression derives a convex estimator of $f$
 by solving a least square problem. In one dimension, the theory of convex
 regression is well established; see \citet{hanson1976consistency} for
the 
consistency result and \citet{Mammen} and \citet{Groeneboom} for the 
rate
of convergence. The consistency of convex regression has been shown in
\citet{Eunji}, and in 
\citet{Seijo} in the multi-dimensional case where the
observations are independent.

To show the consistency of the estimator in our method, we extend the
results in convex regression literature to the Markov processes. Let $X=(X_t
\colon t\geq 0)$ be a positive Harris recurrent and the noise sequence be
a correlated sequence satisfying suitable technical assumptions. We show
that the estimator is converging to the projection of $f$ onto the cone of
convex functions in the Hilbert space of measurable functions. \citet{Eunji}
studied the behavior of the estimator when the model is mis-specifed so
that the function $f$ is non-convex under the much stronger assumption that
the
function $f$ is bounded. Our result relaxes this assumption.

 Recently,  \citet{hannah2011approximate, hannah2011multivariate}
 employed the notion of fitting convex functions in solving dynamic
 programming problems. The key differences between our work and
 \citet{hannah2011multivariate, hannah2011approximate} are as follows:

\begin{enumerate}[i.)]
\item Our method is fully non-parametric while their approach is
semi-parametric and required adjusting 
several parameters before fitting
a convex function or determining the prior distribution for Bayesian updating.
\item \citet{hannah2011approximate, hannah2011multivariate} used the value
iteration method which involves generating many sample paths. In 
contrast,
we are using 
single or two sample paths.
 \item It is well known 
that 
value iteration type algorithms often lead
 to errors that grow exponentially
in the problem horizon. Small local changes at each iteration can lead
to a large global error of the approximation; see Section IV of \citet
{tsitsiklis2001regression} and \citet{powell2009}.
In contrast, in our method the projection to the convex set occurs
asymptotically with respect to the stationary distribution of the underlying
Markov chain 
and has a convergence 
guarantee.
\end{enumerate}

 The literature on approximate dynamic programming (ADP) is also related to
 our work.  Some recent works in this area suggest that the performance of
 parametric ADP algorithms is improved by exploiting structural properties;
 see \citet{wang2000solving, cai2010stable, cai2012dynamic2, cai2012dynamic,
 cai2012shape, cai2013nonlinear}. In addition, \citet{powell-godfrey}
 and \citet{powell2004} consider the cases where the value functions
 are known to be convex and approximate the value function by separable,
 piecewise linear functions of one variable. In \citet{Topaloglu_monotone},
 the monotonicity of  value functions are used to approximate the value
 function where the state space is finite.

  In greater detail, we make the following contributions:
\begin{enumerate}[i.)]
\item We rigorously develop a fully non-parametric method to estimate shape
constrained value functions of multi-dimensional continuous state space
M.C. In the case that the value function is convex, the estimator can be
represented as a piecewise linear function and evaluated at each point in
linear time.
\item We extend the convex regression to the case in which explanatory
variables are sampled along a Markov chain path. Moreover, the observations
are correlated and generated along the same path.
\item We identify the behavior of the estimator in the case of
mis-specification, where the value function is not-convex.
\item We show the convergence of the estimator to the solution of the
projected Bellman equation as the length of the sample path goes to infinity,
\item We extend the non-parametric method to estimate the value functions
which are Lipschitz or monotone and convex.
\end{enumerate}
The rest of this section is organized as follows: In Section 2, we
precisely introduce the mathematical framework for our analysis. In
section 3, we describe our methods. Section 4 presents the extension of
multi-dimensional convex regression to the Markov processes and shows the
consistency of convex regression in this general framework. In Section 5,
we use the results of Section 4 to prove the convergence of our methods.
In Section 6,  we extend our methods to estimate the value functions by exploiting
 other shape structures. In Section 7, we study the efficacy of our methods by applying them to a pricing problem in energy market. 
\section{Formulation}

Let $X=(X_t\colon t\geq 0)$ be a discrete time Markov chain evolving on
a general continuous state space $\mathcal{X}$ embedded on $\bbr^d$. Each
random variable $X_t$ is measurable with respect to the Borel $\sigma$-algebra
associated with $\bbr^d$. The transition probability of the Markov chain
$P(x,B)$ represents the time-homogeneous probability that the next state
will be $X_{t+1}\in B$ given that the current state is $X_t=x$. Let $r(X_t)$
be the reward function received at time $t$, and $\ea$ be a discounting
factor with $\alpha>0$.

  The value function, which is the expected infinite horizon discounted
  reward 
for the Markov chain, is given by
$$V^*(x)= \E\left[\sum_{t=0}^\infty \ea[t] r(X_t)\Big\vert\Big.X_0=x\right].$$
According to the Markov property, we have $$V^*(x)= \E \left[r(X_t)+\ea
V^*(X_{t+1})\Big\vert\Big.X_t=x\right].$$
Define the operator $T:\mathcal{L(X)}\rightarrow \bbr$ by
$$(T \phi)(x)=\E\left[r(X_t)+\ea  \phi(X_{t+1})  \Big\vert\Big. X_t=x\right],
$$
where $\mathcal{L(X)}$ is the space of measurable functions over
$\mathcal{X}$. The operator $T$ can be considered as the Bellman operator
for a fixed policy. It is well known that $T$ is a contraction with respect
to the sup norm.
$$\|T \phi-T \phi '\|_\infty\leq \ea \|\phi-\phi'\|_\infty,$$
for every $\phi_1,\phi_2\in\mathcal{L(X)}$. Furthermore, the value function
is the unique fixed point of equation $V^*=T V^*$; see \citet[p.408] {bertsekas2007dynamic}.

Let $\pi$ be a probability measure on $\bbr^d.$ Define
$$\mathcal{L}_\pi^2=\Big\{\phi:\bbr^d\rightarrow\bbr \mbox{ such that
}\|\phi\|_\pi<\infty\Big\}$$
where
$$\|\phi\|_\pi=\biggl(\int_{x\in \bbr^d }\phi^2(x)\pi(dx)\biggr)^{1/2}.$$
Suppose 
that $\mathcal{C}$ is the set of all convex functions over $\bbr^d$
which are measurable with respect to $\pi$. Note that $\C$ is a closed convex
cone over the space of functions $\mathcal{L}_\pi^2$; see \citet{Eunji}. The
projection operator onto the cone $\C$ 
with respect to the measure $\pi$,
represented by $\Proj_\C $, is defined as
\[\Proj_\C (f)=\arg\min_{\phi\in\C}\|f-\phi\|_\pi.\\
\]
 The projection of $f$ onto the convex cone $\C$, denoted by $\bar
 \phi=\Proj_C (f)$, can be characterized by
$$ \langle f-\bar \phi,\phi-\bar\phi \rangle_\pi\leq 0
$$
for every $\phi\in \C$.
\section{Convex Value Exploration}
In this section, we suggest two different methods to approximate
the value function for a given fixed policy by incorporating the
shape constraints. Here, we first focus on the convexity as a shape
constraint. Next, we extend our methods to monotonicity and Lipschitz
constraints. In the first method, we estimate the value function by explicitly
incorporating the shape constraints. In the second one, we improve the
estimator by incorporating the shape constraint and simultaneously taking
advantage of the fact that the function satisfies a specific linear system
of equations. In the subsequent sections, we discuss the convergence of
these methods.

\subsection{Truncated Method}
Let $X=(X_t\colon t\geq 0)$ be the underlying Markov process. Consider a
single sample path of $X$. The total discounted rewards over this sample
path rise to  noisy observations of the value function at these sample
points. Fitting a convex function to these observations gives an estimation
of the value function. Since we truncate the infinite horizon discounted
reward stream to get the noisy observation at each sample point, we call
this method 
the \textit{truncated method}.

Let $X_1, X_2, \ldots, X_{2N}$ be a sample path of the Markov chain with
length $2N$. Also, assume that $R_1,\ldots,R_{2N}$ is the sequence of
corresponding rewards at sample points 
$R_i=r(X_i)$  for $1\leq i \leq
2N$. A noisy observation of $V^*(X_i)$ is thus given by
 \begin{align}
   Y_i^N=\sum _{i=j}^{2N} \ea[(j-i)] R_j \label{t-1}
 \end{align}
for $1 \leq i \leq N$. Observe that $Y_i^N=V^*(X_i)+\epsilon_i$, where
$\E[\epsilon_i|X_i]$ is close to zero if the number of sample points $N$
is sufficiently large.  We can construct an estimator of $V^*(x)$ by
projecting the noisy observations onto the cone of convex functions. The
projection is possible by fitting a convex function to the points
$(X_1,Y_1^N),\ldots,(X_N,Y_N^N)$.
Assuming $V^*(x)$ is a convex function, we use the least squares estimator
(LSE)  to project
the noisy observations onto the cone of convex functions by solving
\begin{align} \label{Proj_c}
\min_{\phi \in \C}\sum_{i=1}^N (Y_i^N-\phi(X_i))^2.
 \end{align}
Since $\C$ is an infinite-dimensional space, this minimization may appear
to be computationally intractable.
However, it turns out that this minimization can be formulated as a
finite-dimensional quadratic
program 
(QP):
\begin{equation}
\begin{aligned}
\label{Proj_Trunc}
 &\displaystyle\min_{p_i,\zeta_i} \sum_{i=1}^N (Y_i^N-p_i)^2\\
  & p_i \geq p_j +\zeta_j^T(X_i-X_j) &\mbox{for every $1\leq i,j \leq N$}.
 \end{aligned}
\end{equation}
In \citet{Eunji}, it is shown that this least square problem has a
minimizer $(p_1,\zeta_1),\ldots$ $,(p_n,\zeta_n)$, 
and any minimizer $\phi_N$
of (\ref{Proj_c}) over $\C$ satisfies $\phi_N(X_i)=p_i$. We defer more discussion of 
solving this optimization problem more efficiently to Chapter 5. We define our estimator $V_N(x)$ as
\[
            V_N(x)=\sup\Big \{\phi(x) : \phi \in \C, \phi(X_i)=p_i,
            i=1,\ldots,N\Big\}.
\]
The function $V_N$ is a convex and finite value function over the convex
hull of the points $(X_1,\ldots,X_N)$. Furthermore, it is straightforward to
show that $V_N$ is a piecewise linear convex function given by
\begin{align}
V_N(X)= \max_{1\leq i \leq N} (p_i+\zeta^T_i (X-X_i)).\label{VNT}
\end{align}
In the next section, we will show that as the sample size $N\rightarrow
\infty$, the estimator $V_N$ converges uniformly to $V^*$ over every
compact set.
\subsection{Fixed Point Projection}
In this section, we improve the previous method by taking advantage
of the fixed point property of the value function in addition to the 
shape
constraint. The rationale of the method is to 
iteratively apply the Bellman
operator $T$ and project to the cone of convex functions. This method provides
an approximation of the value function as the fixed point of the operator
$T$. First, we start with the ideal case, in which we can exactly compute
the expectation with respect to the stationary distribution 
as well as the
projection to the space of convex functions. Next, we explain a numerical
algorithm that approximately follows this ideal iteration procedure.

Here, we assume the value function belongs to the space of measurable
functions $L^2_\pi$ and is convex. In the next section, we study the
behavior of the estimator in the general case where the value function is
not convex. The value function is the fixed point of the operator $T$, so
$V^*=TV^*$. Moreover, by the convexity assumption, $V^*$ is a fixed point
of the projection operator onto the cone of convex functions, and we have
$V^*= \Proj_\C V^*$. Therefore,
$V^*$ is the fixed point of the combination of the operators $T$ and
$\Proj_\C$ and satisfies
 $$
    V^*= \Proj_{\C} T V^*.
 $$
 In the next theorem, we show the existence of such a fixed point as a
 result the of contraction of both operators $T$ and $\Pi_C$.
 \begin{theorem}\label{ideal_fixed}
 Let $\pi$ be the stationary distribution of the Markov chain $X$, and
 $r\in \mathcal{L}_\pi^2$. Then there exists a unique fixed point $\overline
 {V}\in \mathcal{L}_\pi^2$ such that
   $$
     \overline {V}= \Proj_\C T \overline {V}.
 $$
 Moreover, let $V=(V_k\ ;\  k\geq 0)$ be a sequence of functions in the
 convex closed cone $\C\subset \mathcal{L}_\pi^2$, defined by
 \begin{align}
     V_{k+1}=\Proj_\C T V_k.\label{Ideal.V.iter}
     \end{align}
  Then, we have
 $$
 \|V_{k}-\overline {V}\|_\pi \leq  \ea[k] \|V_{0}-\overline {V}\|_\pi .
 $$
 \end{theorem}
  \begin{remark}
  The sequence $V=(V_k\colon  k\geq 0)$ generated in (\ref{Ideal.V.iter})
  does not converge for an arbitrary norm $\|\cdot\|_\pi$. The assumption that
  $\pi$ is the stationary distribution of the underlying Markov chain $X$
  is essential to guarantee the convergence of the sequence. For instance,
  the projection with respect to the sup-norm is not contraction (see Example
  \ref{contour}). For a similar discussion in the context of parametric ADP,
  see \citet{tsitsiklis2001regression}.
  \end{remark}
  \begin{proof}
  First, we show that if $\pi$ is the stationary distribution of the Markov
  chain, then the operator $T$ is a contraction with respect to the norm
  $\|\cdot\|_\pi$. For any two functions $\phi_1,\phi_2 \in \mathcal{L}_\pi^2$, 
  we 
have
  \begin{align*}
 \E_\pi  (T \phi_1- T \phi_2 )^2
  ={}&\E_\pi \Big(\E [\reward+\ea \phi_1(X_{t+1})|X_t] \cr
  &-\E [\reward+\ea  \phi_2(X_{t+1})|X_t]\Big)^2\cr
 ={}&\ea[2]E_\pi \Big (\E [\phi_1(X_{t+1})-\phi_2(X_{t+1})|X_t]\Big)^2\cr
 \leq{}&\ea[2]E_\pi \Big ( \phi_1(X_{t+1})-\phi_2(X_{t+1})\Big)^2\cr
 ={}&\ea[2] \|\phi_1-\phi_2\|_\pi^2.
   \end{align*}
  Moreover, we know that $\Proj_\C$, the projection operator onto the convex
  cone $\C$, is also a contraction with respect to $\|\cdot\|_\pi$ norm;
  see P.26 \citet{convex}.
 More precisely, if $\phi_1,\phi_2 \in \mathcal{L}^2_{\pi}(X)$, then we have
$$\|\Pi_\C \phi_1 -\Pi_\C \phi_2\|_\pi\leq \|\phi_1 -\phi_2 \|_\pi.$$
Note that
$$\E_\pi \Big (r(X_t)+ \ea \E [\phi(X_{t+1})|X_t]\Big )^2 \leq 2\E_\pi
r(X_t)^2 +2\ea[2] \E_\pi \phi(X_{t+1})^2 < \infty.$$
Thus, $T\phi\in \mathcal{L}_\pi^2$ for every $\phi \in \mathcal{L}_\pi^2$.
Therefore,
  $$
     \|\Proj _\C T \phi_1- \Proj_\C T \phi_2\|_\pi \leq \ea
     \|\phi_1-\phi_2\|_\pi.
  $$
  The rest of the theorem follows directly from the Banach fixed point
  theorem.
  \end{proof}
In the rest of this section, we develop a computational method to approximate
the fixed point over the 
cone $\C$ by using simulated trajectories.
Exact computation of  $TV$ is not generally viable. Evaluating $TV(x)$
at any $x\in \mathcal{X}$ involves the computation of the expectation
$E[V(X_{t+1})\mid X_t=x].$ This expectation is
over a potentially high-dimensional or infinite-dimensional space and hence
can pose a computational challenge.  The following proposition provides an
equivalent characterization to the operator $\Proj_\C T$. As a result of
this proposition, it suffices to evaluate $V(\cdot)$ at two sample points
rather than computing $E[V(X_{t+1})\mid X_t=x].$
\begin{proposition} \label{p_2copy} Let  $X_{t+1}$ and $\widetilde X_{t+1}$
be two independent samples of an M.C. at time $t+1$ given $X_t$. Moreover,
define the random variable $H_t$ such that
\begin{align}\label{H-def}
H_t= r(X_t)+\frac \ea 2  \left(V(X_{t+1})+V(\widetilde X_{t+1})\right).
\end{align}
 For every measurable function $V\in \mathcal{L}^2_\pi$, we have
   \[
             \Proj_\C TV=\arg\min_{\phi \in \C}\E_\pi \Big( H_t- \phi(X_t)
             \Big)^2.
    \]
    Therefore, $\Proj_\C TV=\Proj_\C H$.
    \end{proposition}
    \begin{proof}Let $\bar \phi$ be the projection of $TV$ onto the cone
    of convex functions $\C$, which is the minimizer of
        \[\min_{\phi \in \C}\E_\pi (TV - \phi )^2.
    \]
By using the independence of $X_{t+1}$ and $\widetilde X_{t+1}$ given $X_t$,
we obtain
    \begin{align}
   \E_\pi (TV - \phi )^2={}&\E_\pi \Big(E [r(X_t)+\ea
   V(X_{t+1})|X_t]-\phi(X_t)\Big)^2\nonumber\\
    ={}&\E_\pi\Big[ \Big(E [r(X_t)+\ea V(X_{t+1})-\phi(X_t)|X_t]\Big)\\
   &\times \Big(E [r(X_t)+\ea V(\widetilde
   X_{t+1})-\phi(X_t)|X_t]\Big)\Big]\nonumber\\
    ={}&\E_\pi \Big[\Big(r(X_t)+\ea V(X_{t+1})-\phi(X_t)\Big)\nonumber\\
      &\times \Big(r(X_t)+\ea V(\widetilde X_{t+1})-\phi(X_t)\Big)\Big]\nonumber\\
    ={}&\E_\pi\Big[ \Big(r(X_t)+\frac \ea 2  (V(X_{t+1})+V(\widetilde
    X_{t+1}))-\phi(X_t)\Big)^2\nonumber\\
    & - \frac {\ea[2]} 4 \Big( V(\widetilde
    X_{t+1})-V(X_{t+1})\Big)^2\Big]\label{observe1}
    \end{align}
    for every function $\phi \in \mathcal{L}^2_\pi$. Therefore, we can
    conclude that $\bar \phi$ is also the minimizer of the 
    optimization
    \[
    \bar \phi=\arg\min_{\phi \in \C}\E_\pi \left[r(X_t)+\frac \ea 2
    \left(V(X_{t+1})+V(\widetilde X_{t+1}))- \phi(X_t) \right)^2\right]. 
    \]
    \end{proof}
By using the ergodic property of the Markov chains, it is straightforward
to calculate an estimator of $\E_\pi (H_t-\phi_t)^2$. At each time
step $t=1,\ldots,N$, we generate two independent copies $X_{t+1}$ and
$\widetilde X_{t+1}$ given $X_t$. We call 
$(X_1,X_2,\widetilde
X_2,\ldots,X_{N+1}, \widetilde X_{N+1})$ a 
\textit{``two copy sample path''}.
 \begin{figure}[htbp]
\begin{center}
     \caption{\small A \textit{two copy sample path} of length $4$}
     \hspace{4mm}

     \begin{tikzpicture}[->,>=stealth',shorten >=1pt,auto,node distance=2cm,
  thick,main
  node/.style={circle,fill=blue!20,draw,font=\sffamily\small\bfseries}]

  \node[main node] (0) {$X_1$};
   \node[main node] (11) [right of=0] {$X_2$};
  \node[main node] (12) [below of=11] {$\widetilde X_2$};

     \node[main node] (21) [right of=11] {$X_3$};
  \node[main node] (22) [below of=21] {$\widetilde X_3$};

     \node[main node] (31) [right of=21] {$X_4$};
  \node[main node] (32) [below of=31] {$\widetilde X_4$};

   \path[every node/.style={font=\sffamily\small}]

   (0) edge node [left] {} (11)
   edge node [left] {} (12)
    (11) edge node [left] {} (21)
   edge node [left] {} (22)

         (21) edge node [left] {} (31)
   edge node [left] {} (32);
\end{tikzpicture}

\label{2copy}
\end{center}
\end{figure}
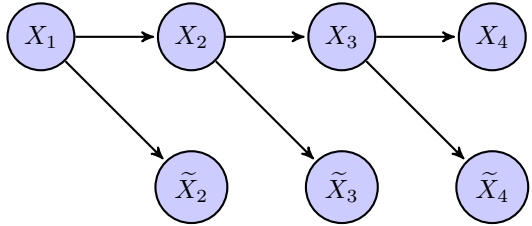

    Under appropriate conditions over the process $X$, we have
    \[
   \frac {1}{N} \sum_{t=1}^N (H_t-\phi(X_t))^2\rightarrow \E_\pi
   (H_t-\phi(X_t))^2
    \]
    as $N\rightarrow \infty$.

Here, we discuss a potential but unsuccessful Monte Carlo  approach to
approximate $V^*$. By the convexity assumption, the value function is the
fixed point of $V^*=\Pi_C T V^*$, and therefore  is the minimizer of the
optimization problem
$$\min_{\phi \in \C}\|\phi-\Proj_\C T  \phi\|_\pi^2.$$
Similar to Proposition \ref{p_2copy}, it is possible to show that the fixed
point $V^*$ is also the minimizer of
$$\min_{\phi \in \C}\E_\pi \Big [\Big( r(X_t)+ \ea \phi(X_{t+1})-
\phi(X_t)\Big)\Big(r(X_t)+ \ea \phi (\widetilde X_{t+1})-
\phi(X_t)\Big)\Big].$$
One might  solve the optimization problem
 \begin{align}
 \min_{\phi \in \C}\frac{1}{N}\sum_{t=1}^N  \Big( r(X_t)+ \ea
 \phi(X_{t+1})- \phi(X_t)\Big)\Big(r(X_t)+ \ea \phi (\widetilde X_{t+1})-
 \phi(X_t)\Big).\label{UBOP}
\end{align}
However, it  can be easily shown that this optimization problem is non-convex
and unbounded for any finite sample path of length $N$; see Example
\ref{unbounded}. We can solve this difficulty by employing an iterative
projection procedure. Before discussing this method, we impose an additional
shape constraint to bound the value function. This assumption helps to
restrict the cone of convex functions and make the projection more tractable.

\begin{assumption} \label{A-bound} Let the state space $\mathcal{X}$ be
bounded. Moreover, assume that for every $x\in \mathcal{X}$, the sub-gradient
of $V^*$ is bounded by a constant $K$:
\[
 \|\nabla V(x)\|_{\infty}<K,
 \] and $V(0)>-K$.
\end{assumption}
\begin{example} Suppose that there exists a constant $K$ such that for
every state $x\in X$ 
we have
\[
\Big|r(x)-E[r(X_1)|X_0=x]\Big|\leq K.
\]
It 
is straightforward to show that
\[
\Big|V^*(x)-\frac{r(x)}{1-\ea}\Big| \leq \frac {K}{(1-\ea)^2}.
\]
Therefore, if the reward function is bounded over the state space, then
Assumption 
(\ref{A-bound}) holds.
\end{example}
Now, we present an alternative method to estimate the value function by
using convexity and the 
fixed point property. The method is similar to the
ideal procedure in Theorem \ref{ideal_fixed}. The main difference is
using the random vector $\hat H^k=(\hat H_1^k,\ldots,\hat H_N^k)$ for a
piecewise linear function $\hat V_k(\cdot)$ instead of $T\hat V_k$. We first
generate a  \textit{two copy sample path} of length $N$. This sample path
does not change throughout the procedure. 
We 
iteratively compute the
random vector $\hat H^k=(\hat H_1^k,\ldots,\hat H_N^k)$ for a piecewise
linear function $\hat V(\cdot)$.  Next, we project $\hat H^k$ onto the
convex cone $\C$ to achieve $\hat V_{k+1}$. Each convex projection is a
least square finite-dimensional 
optimization problem. By following this
procedure iteratively, an estimation of the fixed point over the cone $\C$
is obtained. The details of the method are as follows:
\renewcommand{\thealgorithm}{}
\begin{algorithm}[htb]
\begin{algorithmic}[Fixed Point Projection]\label{FPI}
\INPUT $N$ and $\epsilon$
\OUTPUT The estimator of the value function $\hat V_{k}$. \\
\vspace{5 mm}
\STATE {\bf Initialize:} Select a piecewise-linear function $V_0(x)$,
and  set $k=0$, $V_{-1}(x)=0$.
\vspace{2 mm}
\STATE {\bf Generating Sample Path:} Generate a ``two copy sample path'' of
length $N + 1$.
\vspace{2 mm}
\WHILE {$ \|\hat V_k -\hat V_{k-1}\|_{\tilde{\pi}_N}>\epsilon$}
\STATE
\begin{enumerate}[i.)]
	\item  Compute $(H_t^k)_{t=1}^N$ from  (\ref{forward}).
	\item  Project $(H_t^k)_{t=1}^N$ by solving the optimization problem
	(\ref{approx_opt}), and find $(p_i,\zeta_i)$ for $i=1\ldots, N$.
	\item  Update $\hat V_{k+1}(\cdot)$ thorough (\ref {proj_approx}),
	and $k\gets k+1$.
\end{enumerate}
\ENDWHILE
\RETURN  The piecewise-linear function $\hat V_k$.
\end{algorithmic}
\caption{Fixed Point Projection}
\end{algorithm}

\hspace{2 in}
\begin{description}
\item[Generating Sample Path:] Generate a \textit{``two copy sample path"}
of length $N+1$. At each time step $t=1,\ldots,N$, generate two independent
copies $X_{t+1}$ and $\widetilde X_{t+1}$ given $X_t$.

\item[Updating Step:] Evaluate
\begin{align}
\hat H_t^k=\reward+\frac{\ea}{2} (\hat V_k(X_{t+1})+\hat V_k(\widetilde
X_{t+1}))\label{forward}
\end{align}
 for 
every $t=1,\ldots,N$. The sequence $\hat H^K=(\hat H_1^k,\ldots,\hat
 H_{N}^k)$ is a noisy observation of $T\hat V_k(X_t)$.
\item[Projection:] Project $\hat H^k$ onto the cone of convex functions by
solving 
the 
finite-dimensional convex 
program
\begin{align}\label{approx_opt}
\min \frac{1}{N} & \sum_{t=0}^{N-1}  (\hat H_t^k-p_t)^2 \\
& p_i \geq p_j +\zeta_j^T(X_i-X_j) \mbox  { \ \ \ \  for every $1\leq
i,j\leq N$ }\nonumber\\
&{-K}\leq\zeta_j^l\leq K\mbox  { \ \ \ \	for every $i=1,\ldots,N$, and
$l=1,\ldots,d$ }\nonumber\\
&{-K}\leq p_j +\zeta_j^TX_i.\nonumber
\end{align}
Given 
the optimal solution $(p_t,\zeta_t,X_t)_{t=1}^N$ to this optimization,
we can construct a piecewise linear convex function. Define
\begin{align}\label{proj_approx}
 \hat V_{k+1}(x)=\max_{0\leq i \leq N} (p_i+\zeta_i ^T (x-X_i)).
\end{align}

\end{description}
 The updating and projection stages for a fixed \textit{``two copy sample
 path"} should be continued until a desired level of accuracy is reached. We
 can consider $\hat V_{k}(x)$ as an estimator for the value function. In
 the next section, we will show that for sufficiently large sample size
 $N$ and a large number of iterations $k$, the estimator $\hat V_{k}(x)$
 converges uniformly to the value function $V^*(x)$ over every compact set.\\

\section{Empirical Projection Consistency}
 In this section we describe a generalization of the consistency result of
 convex regression in \citet{Eunji} to the positive Harris chains. Our result
 includes the model mis-specification case without any extra assumption to
 bound the function. In the next section, we use this result to show that the
 estimators in \textit{truncated method} and \textit{fixed point projection
 method} converge to the value function as the sample size grows to infinity.

 Let $X=(X_t\colon t\geq 1) $ be defined on the probability space
 $(\Omega,\mathcal{F},\mathbb{P})$. For every  $\mathcal{F}$-measurable
 random variable $Y$, we can define the projection onto the cone $\C$
 with respect to the norm $\pi$ as the solution of
\[
 \min_{\phi\in \C} \E_\pi(Y-\phi(X))^2.
 \]
 Let $(Y^N _i \colon 1\leq i\leq N,  1\leq N)$ be a sequence of random
 vectors in which $Y^N=(Y_1^N,\ldots,Y_N^N)$ for every $N\geq 1$. We show
 that if $(Y^N \colon N\geq 1)$  \textit{converges on average} to $Y$, then
 the empirical projection of this sequence onto the cone of convex functions
 gives a consistent estimate 
of projecting $Y$ onto this cone. For ease
 of exposition, we define a sequence of random vectors as \textit{strongly
 ergodic} in the following way:
\begin{definition}Suppose that $X=(X_t: t\geq 1) $ is a positive 
Harris
chain with stationary distribution $\pi$, and $(Y^N_t \colon 1\leq N,
1\leq t\leq N  )$ be a sequence of random variables. We call this sequence
\textit{``strongly ergodic''} if there exists a 
$\mathcal{F}$-measurable
random variable $Y$ such that 
$\E Y^2<\infty$, 
\[
\frac{1}{N} \sum_{t=1}^N (Y_t^N-g(X_t))^2I(\|X_t\|\leq
c)\rightarrow\E_\pi(Y-g(X_t))^2I(\|X\|\leq c)\mbox{\ \ \ \ \ \ \ \ \  a.s.}
\]
for every function $g\in \mathcal{L}_\pi^2$, and $c\leq \infty$.
\end{definition}
To illustrate this definition, we provide several examples.
\begin{example} Let the $Y_t^N=f(X_t)+\nu_t$ be such that $\nu_t$
is a sequence of i.i.d noise terms with respect to $X$ such that
$\E\nu_t^2<\infty$, and $f$ is a convex function in $\C$.  Then by the
strong law of large numbers, we obtain that $(Y^N \colon N\geq 1)$ is
\textit{``strongly ergodic''}.
\end{example}
\begin{example}In Lemma \ref{truncated}, we show that if $E_\pi
r(X_t)^2<\infty$, then the two following random sequences are
\textit{``strongly ergodic''}:
\begin{align*}Y_t^N= \sum_{j=t}^\infty \ea[(j-t)] r(X_{j}),\\
\overline Y_t^N= \sum_{j=t}^{2N} \ea[(j-t)] r(X_{j}).
\end{align*}
\end{example}
\begin{example} Assume that $(X_t,\widetilde X_t)$ is a ``two copy sample
path'' of a Harris recurrent chain.  Let  $V \in \mathcal{L}_\pi^2$ and as
we defined in (\ref{forward}),
\[
  H_t^N=\reward+\frac{1}{2} (V(X_{t+1})+V(\widetilde X_{t+1})).
\]
Then, $(H^N\colon N\geq 1)$ is a \textit{``strongly ergodic''}	sequence;
see Lemma \ref{two-ergodic}.
\end{example}
Let $g_N$ be the optimizer of the 
convex optimization problem
\begin{align}\label{approx_projection}
\min_{ g \in \C} \frac{1}{N} & \sum_{t=1}^N  (Y_t^N-g(X_t)^2)
\end{align}
for $N\geq 1$. Note that similar to (\ref{approx_opt}), we can convert this
optimization problem to  a finite quadratic convex problem. In the following
theorem, we show that $g_N$ is an estimator for $g^*$, 
the projection of $Y$
onto the space of convex functions.

We need some assumptions over the structure of the Markov chain.
   \begin{assumption}\label{A3} For the Markov chain $X$, we have:
       \begin{enumerate} [i.)]
	       \item It is positive Harris recurrent with unique stationary
	       distribution $\pi$.
	      \item $\pi (B) >0$ for every positive radius ball $B$ which
	      is a subset of state space $\mathcal{X}$.
	      \item  $E_{\pi}  X_t^2<\infty$.
     \end{enumerate}
\end{assumption}

    \begin{theorem}\label{consistance}
     Assume that $(Y^N \colon N\geq 1)$ is \textit{``strongly ergodic''},
     and 
Assumption (\ref{A3}) holds. Let $g_N$ be the solution of
     (\ref{approx_projection}) and $g^* \in \C$ be the unique minimizer of
     \[
	 \min_ {g \in \C}\E_\pi (Y- g(X))^2.
     \]
     Then,
    \[
     \frac{1}{N} \sum_{t=1}^N(g_N(X_t)-g^*(X_t))^2 \rightarrow 0 \mbox {\
     \ \ \ \  \ \ a.s.}
     \]
     as $N\rightarrow \infty$.
    Moreover,
   $$ \sup_{\|x\|\leq c}  |\hat g_n(x)-g^*(x)|\rightarrow 0\mbox {\ \ \ \
   \  \ \ a.s.}
   $$
   as  $N\rightarrow \infty$, for every $c>0$.
    \end{theorem}
    The proof follows the same steps as the convergence proof in
    \citet{Eunji}. The main difference is the use of the ergodic property of
    Harris chains instead of the strong law of large number for i.i.d random
    variables. Moreover, we continue to allow the model mis-specification
    in which $f(X)=E(Y|X)$ is not a convex function.

    We first start by showing the consistency of the projection onto the
    compact disk $H_c=\{X:\|X\|\leq c\}$ for every $c>0$. Then, by expanding
    this projection over the whole space, we conclude the theorem.

     For every $c>0$, define $\C_c$ as the set of all functions
     $g\in \mathcal{L}_\pi^2$ such that $g$ is a convex function
     over the disc $\{X:\|X\|\leq c\}$. Similar to Proposition 3 in
     \citet{Eunji}, we can show that $\C_c$ is a closed subset of $g\in
     \mathcal{L}_\pi^2$. Therefore, there exists a unique function $g_c^*\in
     \C_c$ which is the projection of $Y$ onto $\C_c$.
     \[\min_{g\in \C_c^2}\E_\pi (Y-g)^2
    \]
    It is clear that $g_c(x)=E[Y|X=x]$ for almost every $x \notin H_c$. In
    Lemma \ref{ggc}, we show $g_c^*$ converges to $g^*$ as $c$ goes to
    infinity.

     \begin{proofof}{ \textit{\autoref{consistance}}}
    Similar to the steps 
1, 2, and 3 in \citet{Eunji}, we have
    \begin{align}
    \frac{1}{N} \sum_{i=1}^N(g_N(X_i)-g^*(X_i))^2 &\leq \frac{2}{N}
    \sum_{i=1}^N(Y^N_i-g^*(X_i)) (g_N(X_i)-g^*(X_i)),\label{main}
    \end{align}
     and for sufficiently large $N$
\begin{align}
      \frac{1}{N} \sum_{i=1}^N(g_N(X_i))^2&\leq  \frac{8}{N}
      \sum_{i=1}^N(Y^N_i-g^*(X_i))^2+\frac{2}{N}
      \sum_{i=1}^N(g^*(X_i))^2\nonumber\\
      &\leq  9E_\pi(Y-g^*(X))^2+3\E_\pi g^*(X_i)^2=\beta.\label{beta}
   \end{align}
   We conclude the last inequality from the \textit{``strongly ergodic''}
   property of $(Y^N\colon N\geq 1)$.  By the 
Cauchy--Schwarz inequality,
   the tail of the empirical inner product can be uniformly bounded for
   every $c>0$ and sufficiently large $N$. Observe 
that
   $$
   \displaylines{
     \frac{1}{N}
     \sum_{i=1}^N(Y^N_i-g^*(X_i))(g_N(X_i)-g^*(X_i))I(\|X_i\|>c)\hfill\cr\hfill
\eqalign{ 
     &\leq \Big (\frac{1}{N} \sum_{i=1}^N(g_N(X_i)-g^*(X_i))^2\Big)^{1/2}
      \Big(\frac{1}{N} \sum_{i=1}^N( Y^N_i-g^*(X_i))^2I(\|X_i\|>c)\Big)^{1/2}\cr
     &\leq  (\beta+2E_\pi (g^*(X))^2)^{1/2}2 \Big(E_\pi (
     Y-g^*(X))I(\|X\|>c)\Big)^{1/2} .}}
     $$
 In the last line, we used the \textit{``strongly ergodic''} assumption
 and the triangle inequality. Since $E_\pi Y^2<\infty$ and $E_\pi
 g^*(X)^2<\infty$,  the right hand side can be smaller than any $\epsilon>
 0$ for large enough $c$. Thus,
 \begin{align}
 \lim_{N\rightarrow\infty}\frac{1}{N}
 \sum_{i=1}^N(Y^N_i-g^*(X_i))(g_N(X_i)-g^*(X_i))I(\|X_i\|>c)\leq \epsilon
 \label{tail}
 \end{align}
  for sufficiently large $c$. Therefore, the terms in (\ref{main}), which
  correspond to the samples outside the disk $\mathcal{H}_c$ can be made
  arbitrarily small. In the next 
lemma, we show that 
$g_N$ converges
  to $g_c^*$ inside the disk.
\begin{restatable}{lemma}{lemrest} \label{restricted} Then 
for every $c>0$,
    \[
    \limsup_{N\rightarrow \infty}\frac{1}{N} \sum_{i=1}^N(Y^N_i-g^*_c(X_i))
    (g_N(X_i)-g^*_c(X_i))I(\|X_i\|<c)\leq0.
    \]
 \end{restatable}
See the Appendix for the proof.
 This lemma ensures that there exists a sequence $\delta_N$ converging to
 zero  such 
that
\begin{align}
\delta_N \geq{}& \frac{1}{N} \sum_{i=1}^N(Y^N_i-g^*_c(X_i))
(g_N(X_i)-g^*_c(X_i))I(\|X_i\|<c)\nonumber\\
={}&\frac{1}{N}
\sum_{i=1}^N\Big((Y^N_i-g^*(X_i))+(g^*(X_i)-g^*_c(X_i))\Big)\times
\nonumber \\
&\ \ \ \ \ \ \ \ \ \ \ \ \ \
\Big((g_N(X_i)-g^*(X_i))+(g^*(X_i)-g^*_c(X_i))\Big)I(\|X_i\|<c)\nonumber\\
={}&\frac{1}{N}
\sum_{i=1}^N(Y^N_i-g^*(X_i))(g_N(X_i)-g^*(X_i))I(\|X_i\|<c)\label{T.1}\\
&+\frac{1}{N}
\sum_{i=1}^N(Y^N_i-g^*(X_i))(g^*(X_i)-g^*_c(X_i))I(\|X_i\|<c)\label{T.2}\\
&+\frac{1}{N}
\sum_{i=1}^N(g^*(X_i)-g^*_c(X_i))(g_N(X_i)-g^*(X_i))I(\|X_i\|<c)\label{T.3}\\
&+\frac{1}{N} \sum_{i=1}^N(g^*(X_i)-g^*_c(X_i))^2I(\|X_i\|<c).\label{T.4}
\end{align}
 We can get a lower bound for (\ref{T.2}) by the 
Cauchy--Schwarz inequality.
$$
\displaylines{
\frac{1}{N} \sum_{i=1}^N(Y^N_i-g^*(X_i))(g^*(X_i)-g^*_c(X_i))I(\|X_i\|<c)
\geq\hfill\cr\hfill -2(E_\pi (Y-g^*(X))^2 )^{1/2} (E_\pi
(g^*(X)-g^*_c(X))^2I(\|X\|<c))^{1/2}}
$$
for 
sufficiently large $N$.
Similarly, we can get a lower bound for (\ref{T.3}) by using (\ref{beta})
and Cauchy-Schwarz inequality
\begin{align*}
\frac{1}{N}
\sum_{i=1}^N(g^*(X_i)-g^*_c(X_i))&(g_N(X_i)-g^*(X_i))I(\|X_i\|<c)\\
&\geq
- \left(\left(\frac{1}{N}
\sum_{i=1}^N(g_N(X_i))^2\right)+ \left(\sum_{i=1}^N (g^*(X_i))^2\right)\right)^{1/2}\\
&\hspace{.2 in}\times \left(\frac{1}{N}
\sum_{i=1}^N(g^*(X_i)-g^*_c(X_i))^2 I(\|X_i\|<c)\right)^{1/2} \\
&\geq -2\Big(\beta+E_\pi (g^*(X))^2\Big)^{1/2} \Big(E_\pi
(g^*(X)-g^*_c(X))^2I(\|X\|<c)\Big)^{1/2}
\end{align*}
By combining these inequalities and using Lemma \ref{restricted}, we obtain
$$
\displaylines{
\frac{1}{N} \sum_{i=1}^N(Y^N_i-g^*(X_i))(g_N(X_i)-g^*(X_i))I(\|X_i\|<c)
\hfill\cr\hfill
\leq \delta(N)+\beta_1 \min\{(E_\pi (g^*(X)-g^*_c(X))^2I(\|X\|<c))^{1/2},1\}.}
$$
According to Lemma \ref{ggc}, there exists a sufficiently large $c$ for
every $\epsilon>0$ such 
that
\[E_\pi (g^*(X)-g^*_c(X))^2I(\|X\|<c)\leq \epsilon'.
\]
Therefore,  there exists a sufficiently large $c$ for any $\epsilon>0$
such that
\[\frac{1}{N} \sum_{i=1}^N(Y^N_i-g^*(X_i))(g_N(X_i)-g^*(X_i))I(\|X_i\|<c)\leq
\epsilon.
\]
Now, we can use (\ref{main}) and (\ref{tail}) to conclude the 
theorem.
\begin{align*}
\lim_{N\rightarrow \infty} &\frac{1}{N} \sum_{i=1}^N(g_N(X_i)-g^*(X_i))^2\\
&\leq\frac{1}{N} \sum_{i=1}^N (Y^N_i-g^*(X_i))(g_N(X_i)-g^*(X_i))I(\|X\|<c)\\
&+\frac{1}{N} \sum_{i=1}^N (Y^N_i-g^*(X_i))(g_N(X_i)-g^*(X_i))I(\|X\|>c)
\leq 2\epsilon.
\end{align*}
 The second part of the theorem is similar to Step 8 in 
\citet{Eunji}.
   \end{proofof}
    \section{Convergence of the Value Function Estimator}
    In this section we show that the estimators given by the \textit{truncated
    method} and the \textit{fixed point projection method} converge to the
    value function as the sample size grows to infinity. The convergence of
    the truncated method holds for a general setting. However, we show the
    convergence of the estimator given by the \textit{fixed point projection
    method} under 
Assumption (\ref{A-bound}) over the value function.
    \begin{theorem}\label{TC1}Let
      \[
    Y_t^N= \sum_{j=t}^{2N} \ea[(i-t)] R_j,
    \]
    and $V_N(x)$ be the estimator of \text{truncated method} defined by
    (\ref{VNT}). Assume that (\ref{A3}) holds. Then, we have
\[
    \sup_{\|x\|<c} |\Proj_\C V^*(x)- V_N(x)|\rightarrow 0,
\]
    where $V_N$ is computed from (\ref{Proj_Trunc}).
    \end{theorem}
    \begin{proof}
    According to Lemma (\ref{truncated}),
    \[
    Y_i^N= \sum_{j=i}^{2N} \ea[(i-j)] R_j
    \]
    is \textit{strongly ergodic}. Therefore,  the result is a direct
    conclusion of Theorem \ref{consistance}.
    \end{proof}
    Now, we are ready to prove the convergence result of the convex iterative
    projection method.	Let
       $$\bar \C_f =\{\phi \in \C: \| \nabla\phi(x)\|_\infty \leq K \mbox{
       for every } x\in \mathcal{X}, \phi(0)\geq -K \}.$$
      In the next theorem, we show that the the estimator (\ref{proj_approx})
      converges to the value function if the value function belongs to $\bar
      \C_f$. In the case that the value function is not convex, the estimators
      converge 
to the fixed point of  $\overline V=\Proj_{\C} T \overline V.$
      The existence of this fixed point is shown in Theorem \ref{ideal_fixed}.

    \begin{theorem}
    Consider a \textit{``two copy sample path''} $X_0,(X_1,\widetilde
    X_1),\ldots,(X_N,\widetilde X_N)$, and  assume that (\ref{A3}) holds. Let
    $(\hat V_k \colon k\geq 1)$ be a sequence of  convex 
functions
    generated by (\ref{forward},\ref{approx_opt},\ref{proj_approx}), and
    $\|\hat V_0\|_\pi<\infty$. Let
    \[
    \overline V=\Proj_{\C} (T \overline V).
    \]
    If $\overline V\in \bar\C_f$, then	there exists a sequence $\beta_N$
    converging to zero such that
    \[
   \frac{1}{N}\sum_{t=0}^{N-1} (\hat V_k(X_t) - \overline V(X_t))^2 \leq
   \ea[k] \gamma_N +\beta_N
    \]
  for sufficiently large $N$, where $\gamma_N= \frac{1}{N}\sum_{t=0}^{N-1}
  (\hat V_0 (X_t)- V^*(X_t))^2$.
 \end{theorem}

\begin{proof}
    First, we give some motivation for the definition of $H_t$ in
    (\ref{forward}). Next, by using the contraction property of projection,
    we show that  the distance between $\hat V_k$ and $\overline V$ is
    approximately shrinking by a factor of $\ea$ at each stage of the
    iteration. 

According 
to the fixed point assumption, $\overline V$ is the minimizer 
of
	\[\min_{\phi \in \C}\E_\pi (T\overline V - \phi )^2.
    \]
    Therefore, due to Proposition \ref{p_2copy}, we can conclude that
    $\overline V$ is also the minimizer of the 
optimization
    \[
    \overline V=\arg\min_{\phi \in \C}\E_\pi ( r(X_t)+\frac \ea 2  (\overline
    V(X_{t+1})+\overline V(\widetilde X_{t+1}))- \phi(X_t) )^2.
    \]
  Observe that $\bar \C_f $ is a closed convex subset of $\C$. Let
    \begin{align}
      \widetilde V_N=\arg\min_{\phi \in \bar \C_f } \frac{1}{N}\sum_{t=0}^{N-1}
      (\tilde H_t -\phi(X_t) )^2,\label{tilde_V}
     \end{align}
     where
     \[
     \widetilde H_t= r(X_t)+\frac \ea 2  (\overline V(X_{t+1})+\overline
     V(\widetilde X_{t+1})).
     \]
	 Suppose $\tilde \pi_N$ is the empirical 
semi-norm induced by
	 sample path $X_0, X_1, \ldots, X_N$. The distance between 
the
	 two functions $f,g$ under  $\tilde \pi_N$ is
    \[
    \|f-g\|_{\tilde \pi_N}= \Big (\frac{1}{N} \sum_{t=0}^{N-1}
    (f(X_t)-g(X_t))^2\Big)^{1/2}.
    \]
    The following lemma asserts that the empirical norm asymptotically
    converges to $L^2_\pi$ norm over $\bar \C_f$.
\begin{restatable}{lemma}{lemuniform}\label{uniform_con}
We have
\begin{align}
\sup_{\phi_1,\phi_2\in \bar \C_f } \Big | \frac{1}{N}\sum_{t=0}^{N-1}
(\phi_1(X_t)-\phi_2(X_t))^2 -\E_\pi (\phi_1(X_t)-\phi_2(X_t))^2\Big
|\rightarrow 0
\end{align}
as $N$ goes to infinity.
\end{restatable}
See the Appendix for the proof.

     Note that $\tilde V$ and $V_{k+1}$ are the projection of $(\tilde H_t)$
     and $(H_t^k)$ onto the convex set $\bar \C_f $ with respect to the
     semi-norm $\tilde \pi$. Since the projection to the convex set is a
     contraction, we 
obtain
\begin{align*}
   \|\hat V_{k+1}- \widetilde V_N\|^2_{\tilde \pi_N}\leq{}&  \|H_k-\tilde H\|_{\tilde
   \pi_N}\\
   \leq{}& \ea/2 \|V_k(X_{t+1})-\overline V(X_{t+1})\|_{\tilde \pi_N}\\
  &+ \ea/2 \|V_k(\widetilde X_{t+1})-\overline V(\widetilde X_{t+1})\|_{\tilde
  \pi_N}.
\end{align*}
By 
Lemma \ref{uniform_con}, we can bound the right hand side by
its expectation and an error term less than $\delta(N)$ in which
$\delta(N)\downarrow 0$ as $N\rightarrow \infty$. 
Thus,
  \begin{align*}
   \|\hat V_{k+1}-\widetilde V_N\|_{\tilde \pi_N}
\leq{}& \ea/2\Big( \E(V_k(X_{t+1})-\overline V(X_{t+1}))^2\Big)^{1/2}\\
  &+ \ea/2\Big( \E(V_k(\widetilde X_{t+1})-\overline V(\widetilde
  X_{t+1}))^2\Big)^{1/2}+\delta_N\\
  ={}&\ea \Big( \E(V_k(X_{t})-\overline V(X_{t}))^2\Big)^{1/2}+\delta_N\\
  \leq{}& \ea \|V_k-\overline V\|_{\tilde \pi_N}+2\delta_N.
\end{align*}
By the triangle inequality, we obtain
  \begin{align*}
  \|\hat V_{k+1}- V^*\|_{\tilde \pi_N}
  &\leq   \|\hat V_{k+1}- \widetilde V_N\|_{\tilde \pi_N}+	\|\widetilde V_N-
  \overline V\|_{\tilde \pi_N}\\
   &\leq \ea \|V_k-\overline V\|_{\tilde \pi_N}+\|\widetilde V_N- \overline
   V\|_{\tilde \pi_N}+2\delta_N.
\end{align*}
From the last inequality, it can be inductively observed that
  \begin{align*}
  \|\hat V_{k}- \overline V\|_{\tilde \pi_N}&\leq \ea[k] \|\hat V_{0}-
  \overline V\|_{\tilde \pi_N}+\frac{1-\ea[(k+1)]}{1-\ea} ( \|\tilde V_N-
  \overline V\|_{\pi}+\delta_N)\\
   \lim_{k\rightarrow \infty}\|\hat V_{k}- \overline V\|_{\pi}&\leq
   \beta_N=\frac{1}{1-\ea} ( \|\tilde V_N- \overline V\|_{\pi}+\delta_N).
\end{align*}
As a result of Theorem	\ref{consistance}, we can show that  $\|\tilde V_N-
\overline V\|_{\tilde \pi_N}$ converges to zero as $N\rightarrow \infty$;
see Lemma \ref{V*-con} for more details. Therefore, 
$\beta_N$ is 
converging
to zero.
    \def\labelenumi{(\roman{enumi})}\end{proof}
    \section{Extensions}
    Here, we consider two extensions to estimate the value functions by
    exploiting other shape structures. First, we consider the case in
    which the value function is Lipschitz for a known constant $K$. The
    second extension employs 
the property that the value function
    is non-decreasing and convex. The main difference between exploiting
    different shape properties is the projection/regression step. In both
    \textit{truncated} and 
{fixed point projection} methods, we can replace
    the projection onto the set of convex functions with projection onto
    the set of Lipschitz or non-decreasing and convex functions. The other
    steps of the methods are similar.
\paragraph{Lipschitz}
 Assume that we know the value function $V^*$ is Lipschitz for a known
 constant $K$. We exploit this property to estimate the value function. In
 particular, we assume that the value function belongs to
 the set
\[
Lip_{K}=\Big\{\phi:\phi \in \mathcal{L}_\pi, |\phi(0)|<\tilde K,
|\phi(x)-\phi(y)|\leq K \|x-y\|_2 \mbox { for all } x,y\in \bbr^d  \Big\}.
\]
We can easily show that $Lip_{K}$ is a closed convex set in the Hilbert
space $\mathcal{L}_\pi$. Projection onto the $Lip_{K}$ can be defined as
\[\Proj_{Lip_K} (f)=\arg\min_{\phi\in {Lip_K}}\|f-\phi\|_\pi.\\
\]
for every $f \in \mathcal{L}_\pi$.
Let $(X_1,X_2,\ldots,X_{2N})$ be a sample path of length $2N$. Let $Y_i^N$
be 
the random variable defined 
in 
Equation (\ref{t-1}) which represents 
a
noisy observation of the value function at the sample point $X_i$. Similar to
the 
\textit {truncated method} for convex value 
functions, the estimator of $V^*$
can be achieved by projecting the random vector $Y^N=(Y_1^N,\ldots,Y_N^N)$
onto the convex set $Lip_{K}$ . The projection is possible by solving the
following QP:
\begin{align} \label{Proj_Trunc_Lip}
  \min_{p_i} \sum_{i=1}^N &(Y_i^N-p_i)^2\\
   p_i -p_j&\leq \  \ K \|X_i-X_j\| \mbox{ \  \  \  \  \  \	\  \   \
   for every $0 \leq i,j \leq N$}\nonumber\\
 -\tilde K&\leq p_0\leq \tilde K, \nonumber
 \end{align}
 where 
$X_0=0$.

Having 
the optimal solution $p_1,\ldots, p_N$ to this optimization, we can
construct an estimator belonging to the set $Lip_{K}$ . Define
\[
V_N(x)=\min_{0 \leq i\leq N} \Big (p_i +K \|x-X_i\|\Big).
\]
Note that the estimator can be evaluated at each point $x$ in linear
time. Similar to Theorem $\ref{TC1}$, it is possible to show that the
estimator $V_N(x)$ uniformly converges to $V^*(x)$ over every compact set
as 
$N\rightarrow \infty$.

Similarly, we can extend the \textit{fixed point projection} to the Lipschitz
case. For a fixed 
\textit{two copy sample path} of length $N$, one can
find $H^N=(H_1^N,\ldots,H_N^N)$ from (\ref{forward}) and next project the
vector $H^N$ to the convex closed set $Lip_K$ by solving a similar QP to
(\ref{Proj_Trunc_Lip}).
\paragraph{Convex and Monotonic.} 
As another extension, we consider the case
that the value function is both convex and non-decreasing. There is a variety
of Markov decision problems in the queue admission, batch service, marketing,
and aging and replacement settings, where the value function is monotone. We
say that a function $\phi:\bbr^d\rightarrow \bbr$ is \textit{non-decreasing}
if $\phi(x)\leq \phi(y)$
whenever $x \leq y$ (so that $x_i \leq	y_i$ for $1 \leq  i \leq  d$). We
now adjust the definition of the cone of functions $\C$ to
\[
\mathcal{CM}=\Big\{\phi:\phi \in \mathcal{L}_\pi, \mbox{and $\phi$ is a
convex and non-decreasing function.} \Big\}.
\]
One can easily show that $\mathcal{CM}$ is a closed and convex cone of the
Hilbert space $\mathcal{L}_\pi$.
Here, we have to project the noisy observations $Y^N$ computed in (\ref{t-1})
for \textit{truncated method}, or $H^N$ computed in (\ref{forward}) for
\textit{fixed point projection} onto the cone $\mathcal{CM}$. The projection
is again obtained by solving a QP:
\begin{align}
 & \min_{p_i,\zeta_i} \sum_{i=1}^N (Y_i^N-p_i)^2\\
  & p_i \geq p_j +\zeta_j^T(X_i-X_j) &\mbox{for
  every $1\leq i,j \leq N$}\nonumber\\
 & \zeta_i\geq 0  &\mbox{for all $1\leq i \leq
 N$}.\nonumber
 \end{align}
In 
addition, the estimator $V_N(x)$ can be evaluated 
by
\[
V_N(x)=\max_{1\leq i\leq N} (p_i+\zeta_i^T (x-X_i)).
\]
Note that for every $x\in \bbr^d$, we have $\partial V_N(x)=\zeta_i\geq 0$
for some $i$. Therefore, the estimator is convex and non-decreasing.

    \section{ Case Study: Pricing Tolling Contracts}
	This section considers the problem of scheduling dual-fuel power
	stations in the presence of switching 
costs. One of the fundamental
	problems encountered in the energy markets is the pricing of tolling
	agreement contracts.  By signing a ``tolling contract'', 
power plant
	owners can reduce their exposure to fuel prices by transferring
	control of the plant to a third party.	This third party is then
	responsible for any costs, fuel or otherwise, involved in meeting
	power plant obligations.  The complexity of pricing such contracts
	arises as a result of interplay between limited flexibility and
	uncertainty.

       Consider a renter who has leased a dual-fuel power plant in a
       de-regulated market. The agent  dynamically determines the operating
       mode of the power plant as the fuel and electricity
prices fluctuate. Our goal is to evaluate the expected total profit 
for
given fixed scheduling policies.

Note that by using the policy iteration method, it is straightforward to
update the policy iteratively and achieve the optimal scheduling policies
as well.

	This specific pricing/control problem is widely considered
	to be a challenging control problem. In mathematical finance
	literature, several authors, including \citet{dixit1989entry,
	Brekke:1994:OSE:181228.181236, doi:10.1080/17442500903106606}, have
	focused on obtaining closed-form solutions by making simplified
	assumptions. The problem also is also studied by \citet{Deng2006,
	carmona2008pricing, djehiche2009finite, Bardou2009} in parametric
	ADP literature.
\subsection{Modeling}
We adopt the model of \citet{carmona2008pricing} in our study. Consider
a dual-mode power plant that can use either natural gas or oil. Due to
increased development of natural gas infrastructure in coastal US regions
in recent years, these power plants have become popular. To run the plant,
the operator buys natural gas or oil,
converts it into electricity and sells the output on the market.

 The fluctuation of prices can be modeled by the gas/oil spark-spread. The
 spark-spread is the 
difference between the price of electricity (output) and
 the prices of its primary fuels (inputs). Specifically, 
let $P_t$ and $G_t$
 be the prices of electricity and gas at time $t$.   The heat rate, denoted by
 $\overline{HR}_G$, is the amount of fuel needed by a power plant to produce
 one kilowatt-hour (kWh) of electricity. The gas spark spread is given by
  $$X_t^1=(P_t - \overline {HR}_G \cdot G_t).$$
Similarly,  the oil spark spread is represented by  $X_t^2=(P_t - \overline
{HR}_O \cdot O_t)$, 
where $O_t$ is the price of oil at time $t$; see \citet
[p. 49-51]{eydeland2003energy} for more details.
 Empirical studies (\citet{eydeland2003energy}) have suggested that the spark
 spread is indeed stationary. We model the driving process $X_t=(X_t^1,X_t^2)$
 as a $2$-dimensional 
Ornstein--Uhlenbeck process with jump, namely
\[
dX_t^n=\kappa (\theta-X_t^n)dt +\Sigma^n \cdot dW_t+Y^n dN_t, \hspace{.5in}
n=1,2,
\]
where 
$W=(t\geq0 \colon W_t)$ is a 2-dimensional standard Wiener process,
$N_t$ is an 
independent Poisson processes with
intensity $\lambda$, $Y^n$ is an independent 
exponential random variable,
and $\Sigma\in \bbr^{2\times 2}$
is a constant non-degenerate volatility matrix.

The mode of operation at each time step $t$ is represented by
$s_t\in\{\mbox{oil,gas}\}$. Also, 
let $s_{-1}$ be the operation mode
immediately before the starting time. Switching is 
allowed only at
the beginning of each time slot. Moreover, changing the operation
mode is costly, requiring extra fuel and various overhead costs. Let
$C_{i,j}^t=c_{i,j}X_t^j$ be the cost of switching from mode $i$ to mode
$j$ if $i\neq j$. Clearly, if there is no switching, the switching cost is
$C_{i,i}=0$ for  $i\in\{\mbox{oil, gas}\}$.

The profit function $\psi(X_t,s_t)$ is considered as a linear function of
spark spread. For instance, if the plant is fueled by natural gas, we define
 $$
     \psi(X_t,\mbox{gas})\df\overline{Cap}_G \cdot (X_t^1 - K_G),
 $$
 where $K_G$ is the operating cost 
and $\overline{Cap}_G$ is the capacity
 of the plant in gas mode.
 Therefore, the value function is
 \[
    V_u(x^1,x^2, i)=\E\left[\sum_{t=0}^\infty e^{-\alpha
    t} \left (\psi(X_t,s_t)-C_{s_{t-1},s_t}^t\right)
    \Big\vert\Big. X_0=(x^1,x^2),s_{-1}=i \right],
 \]
 where $u$ is the switching policy.  In Theorem 3.5.4 of
 \citet{ludkovski2005optimal}, It is shown 
that the optimal value
 function $V_{\mbox{\tiny opt}}(X,i)$ is convex for this model. Furthermore,
 it is straightforward to show that $V_{\mbox{\tiny myopic }}(X,i)$ is a
 convex function for myopic policy.

  Let the current operation mode of the power plant be $s_{t-1}=i$ at the
  beginning of the time slot $t$. Under the myopic policy, the operation
  mode is switched from $i$ to $j$ in the case that
 \[
     \psi(X_t,\mbox{j})-C_{i,j}^t\geq \psi(X_t,\mbox{i}).
 \]
 We numerically compute the value function for the the myopic policy by
 using the convexity property of the myopic policy.
\subsection{Numerical Results}
 In this section, we report our numerical results to estimate the value
 function for the switching problem for a fixed policy. The value function
 is computed for the myopic policy.
  The value function $J(x,s)$ is estimated at the point $x=(10,10),
  s=\mbox{gas}$ by the 
\textit{truncated method} and \textit{fixed point
  projection}.

  The results are compared with the parametric recursive least squares
  method developed for policy evaluation (RLSAPI) in \citet{powell2009}.
  Most of the other methods suggested for solving this problem are based
  on value iteration, and can not compute the performance by a single sample path.

    We need to	specify an appropriate approximation architecture for 
the
    parametric method. Approximation architectures that span polynomials
    are known to work well for switching 
problems. We use all monomials with
    degree at most three which we call the cubic basis 
as our approximation
    architectures.
To have a fair comparison, we compare the result of a two copy sample 
path
of length $N$ with a single sample path of length $2N$ used in RLSAPI.

For solving the optimization $(\ref{Proj_Trunc})$, we use the cutting plane
algorithm as a more efficient approach for solving this optimization problem.

 Table \ref{adp-results} reports the averages (Mean) and the standard
 deviation (Std) of the estimators computed by the truncated
 method, the {fixed point projection, and the RLSAPI.	We wish
 to compute $V_{\mbox{ \tiny myopic}}(x,s)$ at 
$x=(10,10)$ and
 $s=\mbox{gas}$. The results of \textit{truncated method} and the RLSAPI
 method are based on $2000$ replications for each value of $N$.  In the
 \textit{fixed point projection}, we estimate based on $100$ replications
 for each value of $N$.
 Since we approximate the value function at a single point, we can use Monte
 Carlo simulation to compute the value at this point as a benchmark. We
 compute the value of $V_{\tiny \mbox{myopic}}(10,10,\mbox{gas})$ in the
 last row of Table \ref{adp-results} by averaging the discounted reward of
 $M=200,000$ sample paths with length $N=200$.
 \hspace{10mm}
\begin {table}[H]\label{adp-results}
\begin{center}
\scalebox{.90}{
\begin{minipage}{\textwidth}
\begin{tabular}{@{\extracolsep{4pt}}clclclclclclc@{}}
\hline
& \multicolumn{2}{ c }{Truncated Method}& \multicolumn{2}{ c }{Fixed Point
Projection}& \multicolumn{2}{ c }{RLSAPI }\\
 \cline{2-3}\cline{4-5}\cline{6-7}
N & Mean & Std & Mean & Std & Mean & Std \\
 \hline
2000 & 718.22 & 33.94 & 714.16 & 21.88 & 713.81 & 22.66 \\
2500 & 717.33 & 32.11 & 715.96 & 16.92 & 713.61 & 17.62 \\
3000 & 717.70 & 31.95 & 717.48 & 14.81 & 713.53 & 17.36 \\
4000 & 716.84 & 31.56 & 717.12 & 11.17 & 713.96 & 16.32 \\
 \hline
$V_{\tiny \mbox{myopic}}(10,10,\mbox{gas})$ & \multicolumn{6}{ c }{ 716.47} \\
 \hline
\end{tabular}
\caption{\small Performance of the Truncated Method, the Fixed Point
Projection, and the RLSAPI.}
\end{minipage}}
\end{center}
\end {table}

 The parameters of the 
O-U process are set as $\Sigma=[1,0.2;0.2,1],
 \kappa=2,e^{-\alpha}=.9, \theta=10, dt=.1,\lambda=2 $. 
The switching
 cost coefficients are $c_{ \mbox{\tiny gas,oil}}=1$ and $c_{ \mbox{\tiny
 oil,gas}}=2$, and the profit functions are 
 \begin{align*}
 &\psi(x_1,x_2,\mbox{gas})=
 10\cdot (x^1 - 5),\\
 &\psi(x_1,x_2,\mbox{oil})= 15\cdot (x^2 - 6.66).
\end{align*}
Table \ref{adp-results} shows that the performance of the 
truncated method
and fixed point projection are better compared 
to the RLSAPI. 
Furthermore,
the  fixed point projection has less variance compared 
to the
Truncated method.  Finally, it is clear that larger sample sets yield a
significant performance improvement.

\bibliographystyle{apalike}
\bibliography{references}

\newpage
\appendix
\section{Proof Details}

 \begin{restatable}{lemma}{lemggc}
    \label {ggc}  Suppose that $f\in \mathcal{L}_\pi^2$. Let $g^*$ be the
    projection of $f$ onto $\C$, and $g^*_n$ be the projection of $f$ onto
    $\C_n$.  Then, there exists a subsequence $c_k\uparrow\infty$ such that
    \[\|g^*-g^*_{c_k}\|_\pi\rightarrow 0
    \]
    as $k$ goes to infinity.
	\end{restatable}
\begin{proof}
For this lemma, we first observe that the sequence of projected functions are
converging to a convex function. Then, we show this limit is the projection
with respect to 
$\C$.

First, observe that 
$\C_1\supset \C_2\supset\cdots\supset \C$ . Let
$f(X)=\E[Y|X]$, then $f\in \mathcal{L}_\pi^2$.	The projection of $f$
onto $C_n$ is $g_n^*$ and $g_m \in \C_n$ for every $m\geq n$. So we 
have
$$
\displaylines{
\langle f-g_n^*,g_m^*-g_n^*\rangle_\pi\leq 0,\cr
\|f-g_m^*\|^2_\pi+\|g_m^*\|_\pi^2=\|f\|_\pi^2<\infty.}
$$
It is easy to show that
\[
\|f\|_\pi^2\geq \|f-g_m^*\|^2_\pi\geq \|f-g_n^*\|^2_\pi+\|g_m^*-g_n^*\|^2_\pi.
\]
for every $m>n$. Therefore, $\|f-g_m^*\|^2$ is an increasing bounded
sequence. It follows that $W=\lim_{m\rightarrow \infty}|f-g_m^*\|^2_\pi$
for 
some $W<\|f\|_\pi^2$.  Since $\|f-g_m^*\|^2_\pi-\|f-g_n^*\|_\pi^2 \geq
\|g_m^*-g_n^*\|_\pi^2$ for $m\geq n$,  $(g_n^*)$ is a Cauchy sequence in
$L_\pi^2$. Therefore, $g_n^*$ converges to $g_\infty$ in $L^2$ norm. This
implies that there is a sub-sequence $g_{n_k}^*$ converging almost 
surely
to $g_\infty$.

 Let $x,y\in \bbr^d$, and $0<\theta<1$.  Then,
\[
g_{n_k}^*(\theta x+(1-\theta y) y)\leq \theta
g_{n_k}^*(x)+(1-\theta)g_{n_k}^*(y)
\]
for 
every $n_k>\max(x,y)$. Since, $g_{n_k}^*$ converges almost surely to
$g_\infty$ as $n_k$ goes to infinity, $g_\infty$ is a convex function in
$L^2_\pi$, from which we conclude that $g_\infty \in \C$.

Now, we show that $g_\infty$ is the projection of $f$ over $\C$. Every
convex function $\phi \in\C$ is convex over $\mathcal{H}_c$, so $\phi \in
\C_m$ for every $m$. Thus
\[\langle f-g_n^*,\phi-g_n^*\rangle_\pi\leq 0.
\]
Since $\|g_n^*-g_\infty\|_\pi\rightarrow 0$, it easy to show that
\[\langle f-g_\infty,\phi-g_\infty\rangle_\pi\leq 0.
\]
This equality holds for every $\phi \in \C$. As a result, we can conclude that
$g_\infty$ is the projection of $f$ over $\C$. According to the convexity of
$\C$, the projection of $f$ is unique and equal to $g^*$. Thus, $g_\infty=g^*$
almost everywhere, and we 
have
$$\|g_n^*-g^*\|_\pi\rightarrow 0.$$
\end{proof}
\lemrest*
\begin{proofof}{\textit{\autoref{restricted}}}
Replacing the $g_N$ with a fixed convex function $\tilde g$ 
not
dependent  on $N$, showing the lemma is straightforward.  The 
right-hand
average was converging to the inner product $\langle Y-g_c,\tilde
g-g_c\rangle_\pi$. Since, $g_c$ is projection of $Y$ to the close convex
set $\C_c$, this inner product is negative. However, this argument fails
for $g_N$ 
since it depends on $N$. To fix this difficulty, we show that
it is possible to approximate every function $g_N$ by a member of a finite
set of convex functions over $\mathcal{H}_c$. So, the limit is bounded with
a 
corresponding limit for a fixed convex function 
which is asymptotically
negative according to the projection property.

By 
Assumption \ref{A3}, we have $\pi(X\in B)>0$ for every compact disc $B$,
and (\ref{beta}) ensures that
\[
\frac{1}{N} \sum_{i=1}^N g_N(X_i)^2 \leq \beta.
\]
Similar to 
Proposition 4 in \citet{Eunji}, it is possible to show that
for each $c>0$, there exists a 
deterministic $\gamma(c)$, such that $g_N$
is Lipschitz over $\mathcal{H}_c$ with factor $\gamma(c)$ for 
sufficiently
large $N$. It follows that for every $\epsilon> 0$, there exists a finite
collection of convex functions $h_1,h_2,\ldots,h_m$ which is $\epsilon$-net
for $\C_{c(1+\delta)}$; 
for every large $N$ there exists some $h_k$ such that
\[\sup_{x\in\mathcal{H}_c} |g_N(x)-h_k(x)|\leq \epsilon;
\]
 see 
Theorem 6 of \citet{Bronshtein}.
If $h_k$ 
and $g_N$ satisfy this property, observe 
that
$$
\displaylines{
\frac{1}{N} \sum_{i=1}^N(Y^N_i-g^*_c(X_i))
(g_N(X_i)-g^*_c(X_i))I(\|X_i\|<c)
\hfill\cr\hfill\eqalign{
\leq{}& \frac{1}{N} \sum_{i=1}^N(Y^N_i-g^*_c(X_i))
(g_N(X_i)-h_k(X_i))I(\|X_i\|<c)\cr
&+ \frac{1}{N} \sum_{i=1}^N(Y^N_i-g^*_c(X_i))
(h_k(X_i)-g^*_c(X_i))I(\|X_i\|<c)\cr
\leq{}& \sup_{x\in\mathcal{H}_c} |g_N(x)-h_k(x)|\frac{1}{N}
\sum_{i=1}^N|Y^N_i-g^*_c(X_i)|I(\|X_i\|<c)\cr
&+\max_{1\leq k\leq m}\frac{1}{N} \sum_{i=1}^N(Y^N_i-g^*_c(X_i))
(h_k(X_i)-g^*_c(X_i))I(\|X_i\|<c)\cr
\leq{}& \epsilon \cdot \Big(\frac{1}{N}
\sum_{i=1}^N|Y^N_i-g^*_c(X_i)|^2I(\|X_i\|<c)\Big)^{1/2}\cr
&+\max_{1\leq k\leq m}\frac{1}{N} \sum_{i=1}^N(Y^N_i-g^*_c(X_i))
(h_k(X_i)-g^*_c(X_i))I(\|X_i\|<c)}}
$$
Because 
$h_k$'s are bounded over $\mathcal{H}_c$, and $(Y^N_i)$ 
is
\textit{``strongly ergodic''},
$$
\displaylines{
\frac{1}{N} \sum_{i=1}^N(Y^N_i-g^*_c(X_i))
((h_k(X_i)-g^*_c(X_i)))I(\|X_i\|<c)\rightarrow
\hfill\cr\hfill
\E_\pi (Y-g^*_c(X))(h_k(X)-g^*_c(X))I(\|X\|<c)\leq 0}
  $$
as $N\rightarrow \infty$. Here, we use the fact that $g^*_c$ is the projection
of $Y$ to 
$\C_c$, and $h_k$ is also a bounded convex function over
$\mathcal{H}_c$. As a 
result,
$$
\displaylines{
       \lim_{N\rightarrow \infty} \frac{1}{N} \sum_{i=1}^N(Y^N_i-g^*_c(X_i))
       (g_N(X_i)-g^*_c(X_i))I(\|X_i\|<c) \leq 
\hfill\cr\hfill
       \epsilon \Big   (E_\pi
       |Y^N_i-g^*_c(X_i)|^2I(\|X_i\|<c)\Big)^{1/2}}
       $$
for every $\epsilon>0$. This ensures the lemma.
\end{proofof}

\begin{lemma} \label{truncated} Let $X= (X_t\colon t\geq 0)$ be a Harris
ergodic chain. Assume that $E_\pi r(X_t)^2<\infty$. Then the 
following two
random sequences are 
\textit{``strongly ergodic''}:
\begin{align*}
Y_t&= \sum_{j=t}^\infty \ea[(j-t)] r(X_{t})\\
\overline Y_t^N&= \sum_{j=t}^{2N} \ea[(j-t)] r(X_{t}).
\end{align*}
\begin{proof}
Since $X$ 
is a Harris ergodic chain, there exists a stationary process $\hat X$
initialized by invariant measure $\pi$ and a finite coupling time $T$ such
that $X_t=\hat X_t$ for all $t\geq T$; see 
Proposition 3.13 in \citet[Chapter
VII]{asmussen2003applied}. Let
$$
\hat Y_t= \sum_{j=t}^\infty \ea[(j-t)] r(\hat X_{t})
$$
for $t\geq 0$. For every $g\in \mathcal{L}_{\pi}^2$, we have
\[
\frac{1}{N} \sum_{t=1}^N (\hat Y_t-g(\hat X_t))^2I(\|\hat X_t\|\leq
c)\rightarrow\E_\pi(\hat Y-g(\hat X_t))^2I(\|X\|\leq c)\mbox{\ \ \ \ \ \
\ \ \  a.s.}
\]
by the 
Birkhoff--Khinchin 
theorem; see 
Corollary 6.23 in
\citet[p.115]{breiman1992probability}.
Moreover, we can easily show that $$\frac{1}{N} \sum_{t=1}^T \left |(\hat
Y_t-g(\hat X_t))^2-(Y_t-g(X_t))^2\right| I(\|\hat X_t\|\leq c)\rightarrow 0$$
as $N\rightarrow \infty$. Therefore, by using the fact that $X_t=\hat X_t$
for all $t\geq T$, we conclude that
\[
\frac{1}{N} \sum_{t=1}^N (Y_t-g(X_t))^2I(\| X_t\|\leq c)\rightarrow\E_\pi\left
[( Y-g(X_t))^2I(\|X\|\leq c)\right]\mbox{\ \ \ \ \ \ \ \ \  a.s.}
\]
as $N\rightarrow \infty$. Therefore, the sequence $(Y_t)$ is  
\textit{strongly
ergodic}.

We can also use this property to show $(\overline Y_t^N)$ is \textit{strongly
ergodic}. By applying the triangle inequality, we have
\begin{align}
&\left \vert \sqrt{\frac{1}{N} \sum_{t=1}^N (Y_t-g(X_t))^2I(\| X_t\|\leq c)}-
\sqrt{\frac{1}{N} \sum_{t=1}^N (\overline Y_t^N-g(X_t))^2I(\| X_t\|\leq c)}
\right \vert\nonumber \\
&\leq  \sqrt{\frac{1}{N} \sum_{t=1}^N (Y_t-\overline Y_t^N)^2}=
\sqrt{\frac{1}{N} \sum_{t=1}^N \left(\sum_{j=2N+1}^\infty  \ea[(j-t)]
r(X_j)\right)^2} \nonumber\\
&=\ea[(N+1)]\sqrt{\frac{1}{N} \sum_{t=1}^N  \ea[2(N-t)]}\cdot
\left(\sum_{j=2N+1}^\infty  \ea[(j-(2N+1))] r(X_j)  \right).\label{RHE}
\end{align}
We 
can show that the right hand side of (\ref{RHE}) converges to zero. It
is clear that $$\sqrt{\frac{1}{N} \sum_{t=1}^N	\ea[2(N-t)]}\rightarrow 0$$
as $N\rightarrow \infty$. Let
 \[
 U_N\df \sum_{j=2N+1}^\infty  \ea[(j-(2N+1))] r(X_j) .
 \]
According to the assumption that $X$ is a 
Harris ergodic chain and  Theorem
3.6 of \citet[Chapter VII]{asmussen2003applied}, we have
\[
\E[U_N|X_0=x]\rightarrow \E_\pi [V(X_t)]<\infty
\]
as $N\rightarrow \infty$. By using this fact and applying the 
Borel--Cantelli
Lemma, it is straightforward to show that $\ea[(N+1)]U_N$ goes to zero
almost surely.
Observe that
\begin{align*}
\sum_{N=1}^\infty \P_x(\ea[(N+1)]U_N\geq \epsilon)\leq \sum_{N=1}^\infty
\frac{\ea[(N+1)]}{\epsilon}\EE_x[U_N] <\infty
\end{align*}
for 
every $\epsilon>0$. Therefore, $\P_x(\ea[(N+1)]U_N>\epsilon\ ;\  i.o.)=0$,
and the right hand side of (\ref{RHE}) converges to zero almost surely
as $N\rightarrow \infty$. Hence, $(\overline Y_t^N)$ is \textit{strongly
ergodic}.
\end{proof}
\end{lemma}
\begin{lemma}\label{two-ergodic}Assume that $(X_t,\widetilde X_t)$ is a
\textit{two copy sample path} of a positive Harris recurrent chain, and
$E_\pi( r(X_t)^2+V(X_t)^2 )<\infty $. Let
\[
  H_t=\reward+\frac{1}{2} (V(X_{t+1})+V(\widetilde X_{t+1})).
\]
Then, $H=(H_t \ : \ t\geq 0)$ is a \textit{strongly ergodic} sequence.
\end{lemma}
\begin{proof}
The proof is based on constructing a positive Harris recurrent chain for 
the
\textit{two copy sample path} and using the ergodic property.
By the assumption, the Markov chain $X$ is positive Harris
recurrent. Therefore, there exists a regeneration set $R$ such that:
\begin{enumerate}[i.)]
\item Letting $\tau_R=\inf\{t\geq 0 : X_t\in R\}$, we have $\P_x\left
(\tau_R<\infty\right)=1$ for all $x\in \mathcal{X}$.
\item For some $r>0$, $1>\epsilon >0$ 
and some probability measure $\lambda$
on $\mathcal{X}$,
$$
\P^r(x,B)\geq \epsilon \lambda (B), \  \   \  \  \  x \in R
$$
for 
all $B \in \mathcal {F}$; see \citet[p.198]{asmussen2003applied}.
\end{enumerate}
Define the Markov chain $\overline{X}=\left ( (X_t,X_{t+1},\widetilde
X_{t+1})\ : \ t\geq0\right)$ over the state space $\mathcal{X}\times
\mathcal{X}\times \mathcal{X}$.
The transition probability of $\overline{X}$ is induced by the structure
of the 
\textit{two copy sample path}:
\[
\P\left(\overline{X}_{t+1}\in B_1\times B_2\times B_3\mid
\overline{X}_{t}=(x,y,z)\right)=
 \left\{
  \begin{array}{l l}
      \P(y,B_2) \P(y,B_3) & \quad \text{if $y \in B_1$,}\\
 0 & \quad \text{if $y\notin B_1$.}
  \end{array} \right.
\]
We 
show that this Markov chain is positive Harris recurrent. Let
 $$\overline{R}\df \mathcal{X}\times R\times \mathcal{X}.$$
 It is clear that the first hitting time of $\overline{R}$ is almost
 surely 
finite. Moreover, we have
 \begin{align*}
 \P^{r+1}\left((x,y,z), B_1\times B_2\times B_3\right)&=\int_{w\in B_1}
 \P^{r} (y,dw) \P(w,B_2)    \P(w,B_3)	\\
& \geq \epsilon \int_{w\in B_1} \lambda(dw) \P(w,B_2)	 \P(w,B_3)\\
&=\epsilon \bar\lambda(B_1,B_2,B_3)
 \end{align*}
for 
the measurable sets $B_1,B_2,B_3$ and $(x,y,z)\in\overline{R}$. So
the measure $\lambda(\cdot)$ is a common component for the regenerative
set $\overline{R}$. Therefore, the Markov chain $\overline{X}$
is positive Harris recurrent. Moreover, it is easy to show that $\bar
\pi(dx)=\pi (dx) \P(x,dy)\P(x,dz)$ is the invariance probability for $X$.

Thus, $H=(H_t \ : \ t\geq 0)$ is a \textit{strongly ergodic} sequence as
a direct result of the strong 
law of large numbers 
for positive Harris recurrent
chains; see 
\citet[Chapter 17, p.416]{meyn2009markov}.
\end{proof}

\lemuniform*
\begin{proof}
The proof is based on covering $\bar \C_f $ by a finite set of
functions. Then, we can apply the ergodic property of the Markov chain over
each member of that set.

 For every constant $c$, 
and $\delta$,
 By 
Assumption (\ref{A-bound}), the functions $\phi \in \bar \C_f $ are
 convex bounded and Lipschitz  over the set $\mathcal{H}_{c}$; see, for
 example, \citet[Page. 165]{vandervaart}.
  It follows that for every $\epsilon> 0$, there exists a finite collection
  of bounded functions $\Gamma=\{h_1, h_2,\ldots, h_m\}$ 
that is an
  $\epsilon$-net for $\bar \C_f $. This 
means for every $\phi \in \bar \C_f$
   there 
exits some $h_k$ such that
$\sup_{x \in  \mathcal{H}_{c}} |\phi(x)-h_k(x)|\leq \epsilon$; see Theorem
6 of \citet{Bronshtein},
and 
for a recent result, see 
\citet{Guntuboyina2012}.

 For every $\phi_1,\phi_2 \in \bar \C_f $, suppose $h_i, h_j\in\Gamma$	are
 in their $\epsilon$ neighborhoods, correspondingly. Let $\Delta \phi=\phi_1
 -\phi_2$, and	$\Delta h=h_i -h_j$. Then, we have
 \begin{align*}
\|\Delta \phi \|_\enorm & \leq \|(\Delta \phi -\Delta h) I(\|X\|<c)\|_\enorm
+ \|\Delta h I(\|X\|<c)\|_\enorm +\|\Delta \phi I(\|X\|>c)\|_\enorm\\
&\leq 2\epsilon+\|\Delta h I(\|X\|<c)\|_\enorm+2\tilde K
\Big(\frac{1}{N}\sum_{t=0}^{N-1}  I(\|X\|>c)\Big)^{1/2}.
 \end{align*}
 Here, we are using the 
Assumption (\ref{A-bound}) that
 $|\phi_1(x)-\phi_2(x)|\leq \tilde K$ 
and $|\Delta \phi(x)-\Delta
 h(x)|<2\epsilon$ for all $\|x\|<c$. Similarly, we have
 \begin{align*}
\|\Delta \phi \|_\pi & \geq \|(\Delta \phi -\Delta h) I(\|X\|<c)\|_\pi -
\|\Delta h I(\|X\|<c)\|_\pi -\|\Delta \phi I(\|X\|>c)\|_\pi\\
&\geq \|\Delta h I(\|X\|<c)\|_\pi-2\epsilon-2\tilde K \Big(\E_{\pi}
I(\|X\|>c)\Big)^{1/2}.
 \end{align*}
 For 
sufficiently large $c$ 
and $N$, we get
$$2K \Big(\frac{1}{N}\sum_{t=0}^{N-1}  I(\|X\|>c)\Big)^{1/2}+2\tilde K
\Big(\E_{\pi}  I(\|X\|>c)\Big)^{1/2}\leq \epsilon.$$
Therefore,
  \begin{align*}
\|\Delta \phi \|_\enorm-\|\Delta \phi \|_\pi & \leq \Big | \|\Delta h
I(\|X\|<c)\|_\enorm- \|\Delta h I(\|X\|<c)\|_\pi \Big |+5\epsilon.
 \end{align*}
A similar argument shows that the right-hand side is also an upper for
$\|\Delta \phi \|_\pi -\|\Delta \phi \|_\enorm$. To summarize, we obtain
\begin{align*}
\Big |\|\Delta \phi \|_\enorm-\|\Delta \phi \|_\pi \Big |\leq
\max_{1\leq i,j\leq m} & \Big |\frac{1}{N}\sum_{t=0}^{N-1}
(h_i(X_t)-h_j(X_t))^2I(\|X_t\|\leq c)\\
&- \E_\pi (h_i(X_t)-h_j(X_t))^2I(\|X_t\|\leq c)\Big|^{1/2}+5\epsilon
 \end{align*}
for every $\epsilon>0$, and 
sufficiently large $c$ 
and $N$. Positive
Harris recurrent assumption of the Markov chain ensures that the second
term converges to zero almost surely. Thus,
 \begin{align*}
 \limsup_{N\rightarrow \infty} \Big |\|\Delta \phi \|_\enorm-\|\Delta \phi
 \|_\pi \Big |&\leq 5\epsilon
 \end{align*}
 for 
any arbitrarily small $\epsilon$. Thus, (\ref{uniform_con}) holds.
\end{proof}
\begin{lemma}\label{V*-con} We have
\[
\lim_{N\rightarrow \infty} \|\overline V-\tilde V_N\|_{\tilde \pi_N}=0.
\]
\end{lemma}
\begin{proof}
   Let $W_N$ be the projection of $\tilde H$ onto the convex cone $\C$. 
Recall 
that $\tilde V_N$ is the projection of $\tilde H$ to the convex
   set $\bar \C_f $.  Since, $\bar \C_f $ is a subset of $\C$, we have
  $$\|\tilde H-W_N\|_{\tilde \pi_N} \leq \|\tilde H-\tilde V_N\|_{\tilde
  \pi_N}.$$
  By Assumption 
(\ref {A-bound}), $\overline V$ belongs to $\bar \C_f
  $. This ensures that
     \begin{align*}
    \|\overline V-\tilde V_N\|_{\tilde \pi_N}^2&\leq \|\tilde H-\overline
    V\|_{\tilde \pi_N}^2-\|\tilde H -\tilde V_N\|_{\tilde \pi_N}^2\\
    &\leq \|\tilde H-\overline V\|_{\tilde \pi_N}^2-\|\tilde H -W_N\|_{\tilde
    \pi_N}^2\\
    &=\|\overline V-W_N\|_{\tilde \pi_N}^2+2\langle \overline V-W_N,
    W_N-\tilde H\rangle_{\tilde \pi_N}\\
    &\leq \| \overline V-W_N\|_{\tilde \pi_N}^2+2\|\overline V-W_N\|_{\tilde
    \pi_N} \|W_N-\tilde H\|_{\tilde \pi_N}\\
    &\leq 3\| \overline V-W_N\|_{\tilde \pi_N}^2+2\|\overline V-W_N\|_{\tilde
    \pi_N} \|\overline V-\tilde H\|_{\tilde \pi_N}.
  \end{align*}
  Theorem 
\ref{consistance} implies that $\| \overline V-W_N\|_{\tilde
  \pi_N}^2$ converges to zero. Therefore, $ \|\overline V-\tilde V_N\|_{\tilde
  \pi_N}^2$ also converges to zero as $N\rightarrow \infty$.
  \end{proof}
  \section{Examples}
  \begin{example}\label{contour}
Consider the line $L=\{(x,y)\mid 16 y=x\}$ as a convex set, and project the
points $x_1=(1,2)$ and $x_2=( -2,-2)$ onto this convex set with respect to the
sup-norm. It can be easily shown that $$\|\Pi_L x_1 -\Pi_L x_2\|_\infty=6.58>
\|x_1-x_2\|_\infty=4.$$
\end{example}

\begin{example}\label{unbounded}
Let $S=(X_1,X_2,\widetilde X_2,\ldots, X_N ,\widetilde X_N)$ be a sequence
of $2N-1$ normal random variables. Suppose that $\widetilde X_m, X_k$
are the largest and the second largest numbers in this sequence.
For any large enough $M$, let
\begin{align*}
    g_M(x) \df \left\{
  \begin{array}{l l}
    1 & \quad x\leq X_k\\
    (M-1) \frac{Y-X_k}{\tilde X_m- X_k}+1& \quad x> X_k.
  \end{array} \right.
   \end{align*}
  Clearly, 
the function $ g_M$ is convex, and we 
have
   \begin{align*}
   &\frac{1}{N}\sum_{i=1}^{N-1} (g_M(X_i)-\alpha g_M(\tilde
   X_{i+1}))(g_M(X_i)-\alpha g(X_{i+1}))\\
   &=\frac{(N-1)(1-\alpha)^2}{N}+\frac{1-\alpha}{N}(1-\alpha M) \rightarrow
   -\infty
   \end{align*}
as $M\rightarrow \infty$. Therefore, minimization problem (\ref{UBOP})
is unbounded with a 
probability of at least $1/4$ for any large $N$.
\end{example}

  \end{document}